\documentclass{article}
\usepackage{amsmath,amssymb,amsthm}
\usepackage{graphics}

\newtheorem{theorem}{Theorem}

\newtheorem{prop}{Proposition}
\newtheorem{ass}{Assertion}

\title  {Chebyshev systems and zeros of a function on a convex curve}

\author {Oleg R. Musin}

\begin{document}
\date{}
\maketitle

%O.R. Musin Received October, 1997

\begin{abstract}
The classical Hurwitz theorem says that if  $n$ first ``harmonics'' ($2n + 1$ Fourier coefficients) of a continuous function $f(x)$ on the unit circle are zero, then $f(x)$ changes sign at least $2n + 1$ times. We show that similar facts  and its converse hold for any function that are orthogonal to a Chebyshev system.
These theorems can be extended for convex curves in $d$-dimensional Euclidean space. Namely, if a function on a curve is orthogonal to the space of $n$-degree polynomials, then the function has at least $nd + 1$ zeros. This bound is sharp and is attained for curves on which the space of polynomials forms classical polynomial and trigonometric Chebyshev systems. We can regard the theorem of zeros as a generalization of the four-vertex theorem.
%The matter is that the radius of curvature of a plane convex curve regarded as a function on a circle satisfies the orthogonality condition for $n = 1$.
There exists a discrete analog of the theorem of zeros for convex polygonal lines which yields a discrete version of the four-vertex theorem.
\end{abstract}

\section{Introduction}

The classical four vertex theorem for an oval (a plane convex closed smooth curve) states that the curvature function on an oval has at least four extrema (vertices) [1].
Recently, a series of works appeared (see [2-6] and others) in which different versions of this
theorem are studied for convex curves in $\mathbb{R}^d$ (vertices, flat points) and similar phenomena are considered (points of inflection, zeros of higher derivatives etc.).
The majority of the results obtained can be interpreted as the lower estimate of the number of zeros of the function which satisfies certain conditions on the curve.

In this paper, we tried to define conditions imposed on the function which guarantee the existence of zeros of the function. The $L_2$-orthogonality of a function to a Chebyshev system on an interval or a circle is one of these conditions. (Note that Chebyshev systems of functions have already appeared in Arnold's paper  [2] for similar purposes.) It turns out that if a continuous function on an interval or a circle is $L_2$-orthogonal to a Chebyshev system of functions of order $n$, then this function has at least $n$ zeros. In this paper, we show that this result is invertible and the lower estimate is attainable for any given Chebyshev system.

These results are applicable to functions on convex curves in $\mathbb{R}^d$. A curve in $\mathbb{R}^d$ is said to be convex if it intersects any hyperplane at no more than $d$ points while taking their multiplicity in consideration. For plane curves, this definition is consistent with the standard concept of convexity, and the properties of convex curves are now actively studied for $d > 2$ (see [2, 7, 8] and others). In this paper, a simple connection is revealed between the concept of convexity and Chebyshev systems. Let us consider the restriction of the space of linear functions in $\mathbb{R}^d$ to a curve. The curve is convex if and only if this space on the curve is Chebyshev space.

The main result of this paper (the theorem of zeros) for convex curves is that if a function on a curve is orthogonal to the space of $n$-degree polynomials, then the function has at least $nd + 1$ zeros. This estimate is exact and is attained for curves on which the space of polynomials forms classical polynomial and trigonometric Chebyshev systems. We can regard the theorem of zeros as a generalization of the four vertex theorem. We see this by observing that the radius of curvature of a plane convex curve regarded as a function on a circle satisfies the orthogonality condition for $n = 1$. There exists a discrete analog of the theorem of zeros for convex polygonal lines which yields a discrete version of the four vertex theorem.

\section{Chebyshev systems and sign changes}

 Everywhere in this paper we understand a zero of a function as a ``stable'' zero, i.e., changes of sign of the function. We now provide a more exact definition. The set of zeros of a continuous function on an interval or a circle is a compact set and every connected component of the set of zeros is a point or an interval. We say that a function changes sign on the connected component of the set of zeros if, in any neighborhood of this component, there exist two points at which the values of the function are of different signs. We consider the number of zeros (sign changes) of a function to be the number of connected components of the set of zeros where the function changes sign. If it is an interval, we select one point of each of these connected components and refer to this collection as the zeros (sign changes) of the function. Let us take the union of all connected components of sign changes, and consider the complement. It obviously follows from the definition of a sign change that on every connected component of the complement, the function is of constant sign, i.e., nonnegative or nonpositive, not identically zero.

The well-known book by Polya and Szeg\"o [9] contains two statements which are closely connected with this idea. Here are their formulations.

 \begin {ass} [9, II.4, problem 140] If n first ``moments'' of the continuous function $f(x)$ on a finite or infinite interval $a < x < b$ are zero,
$$
\int \limits_a^b
f(x)dx=
\int\limits_a^b
xf(x)dx =
\int\limits_a^b
x^2f(x)dx = \dots =
\int\limits_a^b
x^{n-1}f(x)dx = 0,
$$
then the function f changes sign on the interval $(a, b)$ at least $n$ times provided that it is not identically zero.
\end{ass}

 \begin{ass} [Hurwitz theorem, \cite{PS}, II.4, problem 141] If $n$ first ``harmonics'' ($2n + 1$ Fourier coefficients) of the continuous periodic function $f(x)$ with period $2\pi$ are zero,
$$
\int \limits_a^{2\pi}
f(x)dx=
\int\limits_a^{2\pi}
f(x)\cos(x)dx =
\int\limits_a^{2\pi}
f(x)\sin(x)dx =\ldots
$$
$$
\ldots =
\int\limits_a^{2\pi}
f(x)\cos(nx)dx =
\int\limits_a^{2\pi}
f(x)\sin(nx)dx = 0
$$
then, on any interval exceeding $2\pi$, the function $f(x)$ changes sign at least $2n + 1$ times provided that it is not identically zero.
\end{ass}

These two statements are similar and suggest the existence of a more general statement. We shall show later that this is the case.

Let $D$ be an interval $(a,b)$ or a circle defined on $[0, 2\pi]$. A set of continuous functions $f_i : D \rightarrow R$,
$i = 1, 2,\dots, n$, forms a {\it Chebyshev system} on $D$ if, for any constants $c_j$, $j = 1$, $2$,$\dots$, $n$, which are not simultaneously zero, the linear form $\sum^n_{j=1} c_jf_j(x)$ vanishes on D at no more than $n-1$ different points.

If $D$ is a circle, then $n$, which is the dimension of the Chebyshev space (the order of the system), is an odd number. This follows from the fact that the number of zeros (sign changes) is always an even number. In what follows, we shall not distinguish between the case of a circle $S_1$ and the case of an interval $(a, b)$ by assuming that the number of zeros of the function on a circle is even and the order of the Chebyshev system of functions on a circle is always odd.

The following classical polynomial and trigonometric Chebyshev systems of the spaces $C[a, b]$ and $C[S^1]$ respectively are well known:

\begin{equation}
    \{1, x, x^2,\dots, x^n\}
\end{equation}
and

\begin{equation}
          \{1, \cos x,\dots, \cos(n - 1)x, \sin x,\dots, \sin(n - 1)x\}.
\end{equation}

  Let $\rho(x)$ be a continuous function of constant sign on $D$.
 We say that two continuous functions $f (x)$ and $g(x)$ on $D$ are orthogonal with the weight function $\rho(x)$ if
$$
\int \limits_D f (x)g(x)\rho(x)dx = 0.
 $$

  Note that by the hypothesis of Assertion 1, the function $f (x)$ is orthogonal to the Chebyshev system of functions (1) and, in Assertion 2, it is orthogonal to system (2), respectively, with the weight function $\rho(x) \equiv 1$. It turns out that there is a generalization of Assertions 1 and 2.

  \begin{theorem} If the continuous function $f (x)$, $x \in D$, is orthogonal to the Chebyshev system of functions $f_1(x)$, $f_2(x)$,$\dots$, $f_n(x)$ with the weight function $\rho(x)$, then the function $f (x)$ has at least $n$ sign changes on $D$ provided that the function $f (x)\rho(x)$ is not identically zero.
  \end{theorem}

 We have not found a direct formulation of this theorem in literature concerned with Chebyshev systems. Only closely related problems were considered such as annihilating generalized
 polynomials, zeros of orthogonal Chebyshev polynomials, and others (see, e.g., [10, 11]). At the same time, for the case of a circle and $\rho(x) \equiv 1$, the versions of Theorem 1 are well known (see [2, 12] and others). Actually, this theorem is a simple consequence of the result of Krein [10, Statement 4.2, 13, p. 143, Proposition 1].

Suppose that a set is given on $D$ that consists of two subsets of $p$ and $q$ nodes. In order for, among generalized Chebyshev polynomials (linear forms written as $\sum^n_{j=1} c_jf_j(x)$ ) of order $n$,
there should exist a polynomial for which the collection of roots of an even multiplicity should coincide with the subset of $p$ points and those of an odd multiplicity should coincide with a subset of $q$ points, it is necessary and sufficient that the condition $2p + q < n$ is fulfilled.

In order to prove Theorem 1, let us assume the contrary, i.e., that the number of sign changes of the function $f$ is equal to $q,$ where $q<n$. Then, in accordance with the preceding statement, there exists a generalized Chebyshev polynomial $\varphi(x)$ which changes sign only at the points at which the function $f$ changes sign. Hence, the function $\rho(x)f (x)\varphi(x)$ is of constant sign on $D$, and therefore the integral of this function cannot be zero. The latter fact contradicts the condition of orthogonality of the theorem.

The proof of Assertions 1 and 2 is carried out in the book [9] by a complete analogy. If we denote by $x_1, x_2,\dots, x_q$ the zeros of the function $f (x)$ on $D$, then we have
$$\varphi(x) = (x-x_1)(x-x_2) \dots (x-x_q)$$ in Assertion 1 and
$$\varphi(x) = \sin \frac{x-x_{1}}{2}\cdot \sin\frac{x-x_2}{2}\cdot \dots \cdot \sin\frac{x-x_{q}}{2}$$

Let the number $m(D, n)$ be equal to $n$ if $D$ is an interval and to $n + 1$ if it is a circle. The
number $m(D, n)$ is the minimal possible number of zeros of the function in Theorem 1 if we assume that the function is orthogonal to our Chebyshev system. A question which naturally arises pertains to the invertibility of this theorem. It turns out that the following holds.

 \begin{theorem} [Converse of Theorem 1] Let a Chebyshev system of functions $\{f_i\}$ of order $n$ be given on $D$. If the number of sign changes of the function $f (x)$ on $D$ is no less than $m(D, n)$, then there exists a continuous function $\rho(x)$ of constant sign such that the function f is orthogonal on $D$ to the given Chebyshev system with the weight function $\rho$.
 \end{theorem}

 \begin{proof} It is not hard to give a proof of the theorem using classical methods of the theory of approximations. Konyagin proposed to use the Borsuk-Ulam theorem for carrying out the proof. He also explained to the author that arguments of this kind were used by Tikhomirov in his book [14, Sec. 4.5, Proposition 2].

Let $q$ denote the number of zeros of the function $f$ according to the condition $q \geq m(D, n)$. We are interested in the limiting case where $q = m(D, n)$. Now if $q$ exceeds this value, then we can ``narrow'' $D$ to the interval on which the number of sign changes is limiting and set $\rho(x) = 0$ on its complement.

Let $x_1, x_2,\dots, x_q$ be zeros of the function $f$ on $D$. These points divide $D$ into $n+1$ intervals. Let us consider the arbitrary function $g(x)$ on $D$ which changes sign only at the points $x_i, i = 1, 2,\dots, q$. We now define a map $T : S^n \rightarrow \mathbb{R}^n$ in the following way. For $p \in R_{n+1}$, $p = (t_1, t_2,\dots, t_{n+1})$, $\sum t^2_i = 1$, we consider the step function $C_p$ on $D$, where $C_p$ is equal to $t_i$ on the half-open interval $[x_i, x_{i+1})$. We introduce the notation
$$
y_i = \int \limits_D g(x)C_p(x)f_i(x)dx
$$
and set
$$
T (p) = (y_1, y_2,\dots, y_n) \in \mathbb{R}^n.
$$

According to the Borsuk-Ulam theorem, there exist two antipodes on the sphere, namely, the points $p$ and $-p$, such that $T (p) = T (-p)$. By construction we have $T (p) = -T (-p)$, and therefore $T (p) = 0$.

We shall show that all numbers $t_i$ are of the same sign. Let us consider the sign changes of the function $F (x) = g(x)C_p(x)$ on $D$. If all $t_i$ are of the same sign, i.e., the function $C_p$ is of a constant sign, then the functions $F (x)$ and $g(x)$ change sign on $D$ only at the points $x_i$. Now if the numbers $t_i$ change sign, then the number of sign changes of the function $F$ is smaller than that of the function $g$. However, this cannot be true in accordance with the preceding theorem. Note that the function $F$ on $D$ satisfies the conditions of Theorem 1: $\int_D F (x)f_i(x)dx = 0$. Hence, the number of zeros of $F$ is not smaller than $n$, all numbers $t_i$ are nonzero and of the same sign, and the function $C_p$ is of constant sign on $D$.

Let $g(x) = f (x)\vert f (x)\vert$. Then $g$ has the same zeros as $f$. The function $\rho(x) = C_p(x)\vert f (x)\vert$ is continuous, of constant sign, and vanishing only at the points $x_i$. Thus, it satisfies the conditions of the weight function and necessarily $\rho(x)$ is the required function. We have proved the theorem.
\end{proof}

The proof of Theorem 2 shows us that the bound of the number of sign changes in Theorem 1 can not be improved.

\begin{theorem} For a given collection of $m(D, n)$ points and the Chebyshev system of functions $\{f_i\}$ of order n on D there exists a function which changes sign only at the given collection of points and which is orthogonal to the given Chebyshev system with a weight function $\rho \equiv 1$.
\end{theorem}
\begin{proof}
The proof of this statement repeats almost verbatim the proof of Theorem 2. As the function $g,$ we take an arbitrary continuous function with zeros at the given collection of points. Then the zeros of the function $F (x) = g(x)C_p(x)$ coincide with those of the function $g$ and therefore the function $F$ is orthogonal to the given Chebyshev system with the weight function $\rho \equiv 1$.
\end{proof}

 \section {Convex curves in ${\Bbb R}^d$ and Chebyshev systems} %CONVEX CURVES IN Rd AND CHEBYSHEV SYSTEMS

We say that a curve $K$ in $\mathbb{R}^d$, which is an Euclidean space of dimension $d$, is {\it convex} if it intersects any hyperplane (a subspace of dimension $d-1$) at no more than $d$ points with due account of their multiplicity.

For the notion of the {\it multiplicity} of the point of intersection of a curve and a hyperplane there is no need to require that the curve be smooth. We interpret the multiplicity of the intersection at the chosen point as the largest number of intersection points lying in a small neighborhood of the chosen point after a slight perturbation of the hyperplane. In the sequel, we shall consider the cases of a closed and a nonclosed curve, i.e., the cases where the curve $K$ is an embedding of a circle, $\sigma: S^1 \rightarrow \mathbb R^d$, and an interval, $\sigma : (a, b) \rightarrow \mathbb{R}^d$.

This definition of a convex curve is introduced and studied in the works by Arnold [2], Vargas [7], and others. In geometry, a similar definition of convexity appeared at the beginning of the century and was mentioned in a number of works such as, for instance, the famous book by Burago and Zalgaller ``Geometric Inequalities''. %(Leningrad: Nauka, 1980). %However, only due to the works by Arnold, this definition became well known in the last years.

We denote by $\Pi_n,d$ the space of all polynomials of degree not higher than $n$ on $\mathbb{R}^d$ and let $K$ be a curve embedded into $\mathbb{R}^d$. We shall consider the restriction of the space $\Pi_{n,d}$ to $K$ and denote the obtained space of functions on $K$ by $\Pi_{n,d}\vert K$.

We are interested in the case where the obtained space is a Chebyshev space, i.e., the functions on $K$ constituting the bases of the space $\Pi_{n,d}\vert K$ form a Chebyshev system. Theorems 1, 2, and 3 are applicable in the case of curves $K$ and functions on $K$ of this kind.

According to Theorem 1, if the function $f$ on $K$ is orthogonal to the space $\Pi_{n,d}\vert K$, then the number of its zeros is not smaller than dim $\Pi_{n,d}\vert K$. According to Theorem 3, this estimate can not be improved. Note that the parametrization of the curve is irrelevant here. The choice of different parameters of the curve will give different weight functions.

Consider the case $n = 1$. Here the dimension of the space $\Pi_{1,d}$ is $d + 1$. The dimensions of the spaces $\Pi_{1,d}$ and $\Pi_{1,d} \vert K$ coincide if and only if $K$ does not lie entirely in some hyperplane $\mathbb{R}^d$. For curves of this kind, the condition that the restriction of the space of polynomials to $K$ is Chebyshev space coincides with the condition that the curve $K$ is convex. This follows from the fact that any linear function on $\mathbb{R}^d$ (a polynomial of degree 1) uniquely defines the hyperplane on which it is zero and this hyperplane cannot intersect $K$ at more than $d$ points if $K$ is convex, and, at the same time, this condition is exactly the same condition under which the space of the restriction of the linear functions to $K$ is a Chebyshev space. We thus have obtained the following theorem.

 \begin{theorem} If the curve $K$ embedded into the Euclidean space $\mathbb{R}^d$ does not lie in any one of its hyperplanes, then the space is a Chebyshev space if and only if $K$ is a convex curve.
 \end{theorem}

For $n > 1$, the space $\Pi_{n,d}\vert K$ may not be Chebyshev space.

\medskip

\noindent{\bf Example 1.} Let us consider, in the plane, a convex curve $K$ which intersects some circle at more than four points. It is easy to construct such a curve. Consider a regular $m$-gon and a circle centered at the center of the $m$-gon with radius smaller than that of the circumscribed circle of the polygon, but larger than that of the inscribed circle. The circle will intersect the polygon at $2m$ points. By a slight perturbation, we can turn the polygon into a smooth convex curve which will intersect the circle at $2m$ points as before. For this curve $K$, the space $\Pi_{2,2}\vert K$  is not a Chebyshev space since if it were Chebyshev, then its dimension would not exceed $5$ and, hence, any second-order curve would intersect $K$ at no more than four points.

\medskip

The following natural question arises: for what class of curves in $\mathbb{R}^d$ will the space $\Pi_{n,d}\vert K$ be a Chebyshev space? We have no answer for the case $n>1$. For $n = 1$ the answer is given by Theorem 4.

\medskip

\noindent{\bf Example 2.} It is interesting that if we consider $P_{n,d}$, which is the space of homogeneous polynomials of degree n on $\mathbb{R}^d$, and a curve $K$ such that the space $P_{n,d}\vert K$ is Chebyshev space, then $K$ may not necessarily be convex. For instance, if $K$ is the graph of the function $y = f (t)$, $t \in (a, b)$, in the plane, then the condition that the space be Chebyshev is fulfilled if $\frac{f(t)}{t}$ is a monotonic function on $(a, b)$. In particular, if $f(t) = \sin(t) + c$, $0 < a < b < c-1$, then this condition is fulfilled and the curve $K$ is not convex if $b-a > \pi$.

\medskip

We can give a large number of examples of convex curves where $\Pi_{n,d}\vert K$ is a Chebyshev space.

\medskip

\noindent{\bf Example 3.} Let the curve $K$ in $\mathbb{R}^d$ be defined as
$$x(t) = (t, t^2,\dots, t^d), t \in (a, b). $$
The curve $K$ is convex and $\Pi_{n,d}\vert K$ is a Chebyshev space of dimension $nd + 1$. This follows from the fact that the basis of this space, $1, t, t^2,\dots, t^{nd}$, forms a classical polynomial Chebyshev system. According to Theorem 1, the number of zeros of the function that is orthogonal to this system is not smaller than $nd + 1$. Due to Theorem 3, this estimate is attainable.

\medskip

\noindent{\bf Example 4.} Let $$x(t) = (\cos(t),\dots, \cos(kt), \sin(t),\dots, \sin(kt)), t
\in [0, 2\pi],$$ where $d = 2k$ and $K$ is a convex closed curve in $\mathbb{R}^d$. The dimension of the Chebyshev space $\Pi_{n,d}\vert K$ is $nd + 1$. If the function $f$ on $K$ is orthogonal to this space, then the number of its zeros is not smaller than $nd + 2$. This follows from the fact that $K$ is an embedding of a circle.

\medskip

\noindent{\bf Example 5.} It is well known that the collection of arbitrary power functions $\{1, t^{\alpha_1},\dots, t^{\alpha_d}\}$, where $0 < \alpha_1 < \dots < \alpha_d$, forms a Chebyshev system on an arbitrary interval [12]. If we consider the curve $K$ on $\mathbb{R}^d$:
$$x(t) = (t^{\alpha_1} , t^{\alpha_2} ,\dots, t^{\alpha_d} ), t \in (a, b), a > 0,$$
we see that it is convex. This is due to the fact that $\Pi_{n,d}\vert K$ is a Chebyshev space. The dimension of the space $\Pi_{n,d}\vert K$ is not smaller than $nd + 1$. If the set of numbers $\{\alpha_i\}$ is rationally independent, then the dimension of $\Pi_{n,d}\vert K$ coincides with that of $\Pi_{n,d}$.

\medskip

\noindent{\bf Example 6.} Consider the convex curve $K: x(t) = (t, e^t), t \in (a, b)$, in the plane. The basis of the space $\Pi_{n,2}\vert K$ is defined by the set
$$
\{1, t,\dots, t^n, e^t,\dots, e^{nt}, te^t, t^2e^t,\dots, te^{(n-1)t}\}.
$$
This set forms a Chebyshev system [10, 11] of order $\frac{(n+1)(n+2)}{2}$ . According to Theorem 1, every function which is orthogonal to $\Pi_{n,2}\vert K$ has at least $\frac{(n+1)(n+2)}{2}$ zeros, and not $2n+1$ as in Example 2.

\section{Zeros of functions on convex curves} %3. THEOREM OF ZEROS OF A FUNCTION ON A CONVEX CURVE

Examples 3 and 4 in the preceding section show that the number of zeros of the function f on the convex curve $K$ that is orthogonal to $\Pi_{n,d}\vert K$ is not smaller than $nd + 1$ if $K$ is not closed and $nd + 2$ if it is closed. In this case, the estimate for these curves can not be improved. It turns out that a similar estimate is valid for arbitrary convex curves.

\begin{theorem} [on zeros of a function] Let $f$ be a function defined on the convex curve $K$ embedded into $\mathbb{R}^d$ and suppose that on this curve there exists a parametrization $s$ such that
$\int_K P_n(x(s))f(s)ds = 0$ for the arbitrary polynomial $P_n(x)$ in $d$ variables of degree not higher than $n$. Then the function $f$ on $K$ has at least $dn + 2$ zeros if $K$ is a closed curve and $dn + 1$ zeros if $K$ is not closed.
\end{theorem}

Example 1 shows that $\Pi_{n,d}\vert K$ is not a Chebyshev space for any convex curve the space , and therefore we cannot use Theorem 1 in order to prove this theorem. The proof that we propose here is close to that of Theorem 1, although Chebyshev systems are not used in it.

\begin{proof} Assume that the number of zeros of the function f on $K$ is $m$, where $m \leq dn$. Then we can represent $m$ as $m = kd + l$, $0 < l \leq d, k + 1 \leq n$. We divide the set $S$ of $m$ points of the curve at which the function vanishes into $k + 1$ subsets $S_1, S_2,\dots, S_{k+1}$, where each of the first $k$ subsets contains $d$ points and $S_{k+1}$ contains $l$ points.

Let us consider a hyperplane passing through $d$ points of the subset $S_i$, $i < k + 1$, and denote by $L_i(x), x \in \mathbb{R}^d$, the support function of the set $S_i$ (a nonzero linear function which vanishes in the hyperplane). By the definition of a convex curve, this hyperplane has no other points of intersection with $K$ other than $S_i$, and, consequently, $L_i$ vanishes on $K$ only at the points from $S_i$.

Let us show that through any given $\ell \leq d$ points on $K$ we can draw a hyperplane whose support function on $K$ changes sign only at the defined points. If the curve is closed, then we choose one of the defined points, and if it is not closed, then we choose an endpoint of $K$. We mark $d-\ell$ points on $K$ in a small neighborhood of the chosen point. From the obtained $d$ points we can uniquely determine a hyperplane which intersects $K$ only at these points. We turn our attention to the new points and the chosen point and consider the limiting hyperplane. Possibly, it may ``touch'' $K$ at some points besides the given points, but it is clear from the consideration of a limit that the support function of this hyperplane on $K$ can change sign only at $\ell$ given points.

In the case of a closed curve, the multiplicity of the intersection of the limiting hyperplane with $K$ is $d-\ell + 1$. Since convex curves can be closed only in even dimensions [2] and the function always has an even number of zeros on them, the numbers $d$ and $\ell$ are even and, hence, the multiplicity of the chosen point is an odd number, i.e., the function changes sign at it. If the curve is not closed, then the multiplicity of the intersection for all given points is equal to $1$. Thus, the support function changes sign only at the given points. We denote it by $L_{k+1}$.

Let $P (x) = L_1(x)L_2(x) \dots L_{k+1}(x)$. Then $P (x)$ is a polynomial on $\mathbb{R}^d$ of degree $k + 1 \leq n$ which changes sign on $K$ only at points from the set $S$. Since the sets of points where $f$ and $P$ change sign coincide, it follows that $f (x)P (x)$ is a function of constant sign on $K$ and the integral of this function cannot vanish. This contradicts to the hypothesis of the theorem and, hence, there are at least $dn + 1$ zeros of the function $f$ on $K$. For closed curves, the number of zeros must be even. Since convex curves can be closed only for even $d$, the number of zeros for them is not smaller than $dn + 2$. We have proved the theorem.
\end{proof}

Two open questions arise in connection with this theorem:

\medskip

\noindent {\it 1. Is it true that for the arbitrary convex curve $K$ in the space $\Pi_{n,d}\vert K$ there exists a Chebyshev subspace of dimension $nd + 1$?}

\medskip

If it is true, then Theorem 5 is a corollary of Theorem 1.

\medskip

\noindent {\it 2. Is it true that for any arbitrary convex curve $K$ and function $f$ defined on it which is orthogonal to the space $\Pi_{n,d}\vert K,$ the number of zeros of $f$ is not smaller than dim $\Pi_{n,d}\vert K$?}

\medskip

Theorem 1 answers this question affirmatively for Chebyshev spaces. The question is whether it is necessary for the space to be Chebyshev. Possibly, we must consider the case of algebraic curves separately.

\section{Corollaries and the Four-Vertex Theorem} % COROLLARIES OF THE THEOREM OF ZEROS AND THE FOUR-VERTEX THEOREM

According to Theorem 4, $\Pi_{1,d}\vert K$ is a Chebyshev space for $n = 1$. In this case, it follows from Theorem 5 of zeros (or Theorem 1) that the number of zeros of the function which is orthogonal to $\Pi_{1,d}\vert K$ is not smaller than $d + 1$, and $d + 2$ if the curve is closed.

In all generalizations of the Four-Vertex theorem for closed curves in $\mathbb{R}^d$ (see [2, 8, 7]) we have precisely the number $d + 2$. This coincidence is not accidental, and we shall show later that the Four-Vertex Theorem follows from the theorem of zeros (to be more precise, from Hurwitz theorem (Assertion 1) which is a corollary of Theorem 1). It is not quite clear how the results of the papers [2, 8, 7] can be obtained from Theorem 5. These results can be interpreted as the estimate of the least number of zeros (vertices, flat points) of some functions. It follows from Theorem 2 that these functions are orthogonal to the space $\Pi_{1,d}\vert K$ with certain weight functions. It is of interest to find explicit expressions for weight functions and the parametrization of curves.

Below we shall give interpretations for the theorem of zeros for the case $n = 1$. It is convenient to formulate them in the language accepted in mechanics with the use of the concept of the center of mass.

We shall assume that the curve $K$ embedded into $\mathbb{R}^d$ has a mass and that the mass of its arc is in direct proportion to the length of the arc; if $dm$ is the mass of the $ds$ long arc, then $dm = \rho(s)ds$. The continuous positive function $\rho$ is called a function of density of the mass on a curve, and if $\rho$ is a constant, then we say that the mass is uniformly distributed. The center of mass of a curve with the density function $\rho(s)$ is the point $x_c \in \mathbb{R}^d$:

$$
x_c(K, \rho) = \frac{1}{M_\rho}
\int \limits_K
x(s)\rho(s)ds,\text{ } x \in \mathbb{R}^d, \text{   } M_\rho = \int \limits_K
\rho(s)ds.
$$
The quantity $M_\rho$ is the mass of the curve with the density function $\rho$.

Here is a corollary of the theorem of zeros. If the curve is not closed, then the endpoints are considered to be extremal.

\begin{prop} Let $f$ be a continuous function defined on $K$ which is a compact convex curve in $\mathbb{R}^d$ such that the center of mass of the curve with a uniformly distributed mass coincides with the center of mass of $K$ with the density function equal to $f$ . Then for $f$ on $K$ there exist at least $d + 2$ extremal points.
\end{prop}

We can propose a relative version of this statement. If $f$ and $g$ are continuous positive functions on $K$ such that their centers of mass coincide, then the function $f =g$ on $K$ has at least  $d + 2$ extrema.

\begin{proof} We can assume, without loss of generality, that the mass $M_f$ is equal to $M_\rho$, where $\rho = const$ is the density of the uniformly distributed mass. Let $h(t) = f (t)-\rho$, $t\in  K$. Then it follows from the coincidence of the masses and the centers of mass that the function h(t) is orthogonal to $\\Pi_{1,d}\vert K$. According to Theorem 5, this function has at least $d + 1$ sign changes if $K$ is not closed and $d + 2$ sign changes if it is closed. It follows that the function $h$ (i.e., $f$ as well) has at least $d + 2$ extrema.

In order to prove the relative version of the statement, we can show, by analogy, that the function $f-g$ with $Mf = Mg$ on $K$ has at least $d + 1$ zeros. Consequently, the function $f/g-1$ has the same number of zeros and no less than $d + 2$ extrema.
\end{proof}

Using Proposition 1, we can easily derive the Four-Vertex Theorem. If $R(\alpha) = 1/k(\alpha)$ is the radius of curvature at a given point of the oval (a convex smooth closed curve in the plane), $\alpha$ is the angle between the positively directed tangent to the point of the oval and the direction of the $x$-axis, then (see [15])
$$
\int\limits_{-\pi}^{\pi}
R(\alpha) \cos(\alpha)d\alpha = 0\text{,  }
\int\limits_{-\pi}^{\pi}
R(\alpha) \sin(\alpha)d\alpha = 0.
$$

Let $K$ be a unit circle (parameter $\alpha$) parameterized in the standard way, with center at the origin and suppose that a function $f = R$ is defined on it. The conditions imposed on $R$ mean that its center of mass coincides with the center of the circle. Using now Proposition 1, we find that since $d = 2, R$ has no less than four extrema.

We can derive by analogy the following result, which apparently belongs to Blaschke [15]: Let $C_1$ and $C_2$ be two plane (positively directed) convex closed curves, $ds_1$ and $ds_2$ be elements of the arc at points with parallel (similarly directed) support lines; then the ratio $ds_1/ds_2$ has at least four extrema. This theorem becomes the Four-Vertex Theorem for an oval if $C_2$ is a circle.

\section{Discrete Analog of the Theorem of Zeros} %DISCRETE ANALOG OF THEOREM OF ZEROS

We say that the polygonal line $K$ in $\mathbb{R}^d$ is {\it convex}  if every hyperplane which does not pass through the vertices of $K$ cuts the polygonal line at no more than d points. There is a version of the theorem of zeros for a convex polygonal line.

\begin{theorem} (a discrete version of the theorem of zeros). Suppose that at the nodes $x_i\in \mathbb{R}^d, i = 1, 2,\dots, k$, of the convex polygonal line $K$ embedded into $\mathbb{R}^d$ there are numbers $f_i$ such that
$\sum_i P_n(x_i)f_i = 0$ for the arbitrary polynomial $P_n(x)$ in $d$ variables of degree not higher than $n$. Then the numbers $f_1, f_2,\dots, f_k$ change sign at least $dn + 1$ times if $K$ is not closed and $dn + 2$ times if it is closed.
\end{theorem}

It is easy to formulate the discrete analog of Proposition 1. We give it here in a relative form.
\begin{prop}
If $\{f_i\}$ and $\{g_i\}$ are sets of masses at the nodes of the convex polygonal line $K$ such that their total masses and centers of mass coincide, than the quantities $f_i-g_i, i = 1,\dots, k$, change sign at least $d + 1$ times ($d + 2$ times if $K$ is closed).
\end{prop}

The proof of Theorem 6 and Proposition 2 is similar to the proof of Theorem 5 and Proposition 1, with the only difference that the integral must be replaced with a sum.

There exist several versions of the Four-Vertex Theorem for a polygon [5, 6]. Here, we shall consider one of them, namely, {\it Aleksandrov's lemma} from his book [16]. We give it in a somewhat less general form than it is given originally.

\medskip

\noindent{\it Suppose that two convex polygons $M_1$ and $M_2$ are given in the plane whose sides are parallel and whose perimeters are equal. Then, when we traverse the polygon $M_1$, the difference of the lengths of the respective parallel sides must change sign at least four times.}

\medskip

Note that for an arbitrary polygon with sides lengths $a_i$ we have the equality $$\sum_i a_in_i = 0,$$ where $n_i$ is a unit normal vector to the side $a_i$. The polygons $M_1$ and $M_2$ have the same set of vectors $n_i$, and if we consider the polygon formed by the ends of the vectors $n_i$ laid off from the same point and define, at the vertices of the resulting polygon, two sets of numbers equal to the lengths of the respective sides of the polygons $M_1$ and $M_2$, then these sets of numbers will satisfy the hypothesis of Proposition 2 for $d = 2$. Thus, Aleksandrov's lemma is a corollary of Proposition 2.

\medskip

{\bf Acknowledgments.}   % ACKNOWLEDGMENTS.}
I want to express my gratitude to V. I. Arnold who looked through the paper and made useful remarks, to A. L. Garkavi and V. B. Demidovich for useful and constructive discussions of themes related to Chebyshev systems, to S. V. Konyagin for his attention to this paper and for the remarkable, to my knowledge, proof of the inverse theorem, to I. K. Babenko and V. D. Sedykh for useful discussions and remarks.

This paper was written under the support of the Russian Foundation for Basic Research
 (project no. 97-01-00174) and the European Program ESPRIT (grant 21042).

\medskip

O. R. Musin, Chebyshev systems and zeros of a function on a convex curve,
PROCEEDINGS OF THE STEKLOV INSTITUTE OF MATHEMATICS, Vol. 221, 236-246, 1998.

\end{document}